\documentclass[a4paper,11pt]{article} 
\usepackage[a4paper,top=1in,bottom=1in,left=1in,right=1in]{geometry}
\usepackage{graphicx}
\usepackage{amsthm,amssymb,amsmath,latexsym} 
\usepackage[unicode,final=true]{hyperref} 
\usepackage{bm}

\usepackage[algoruled,vlined]{algorithm2e}
\SetKwInOut{Input}{Input}
\SetKwInOut{Output}{Output} 
\newtheorem{theorem}{Theorem}
\newtheorem{proposition}[theorem]{Proposition}
\newtheorem{corollary}[theorem]{Corollary}
\newtheorem{lemma}[theorem]{Lemma}
\theoremstyle{definition}
\newtheorem{definition}[theorem]{Definition}
\newtheorem{remark}[theorem]{Remark}

{\bfseries}{\rmfamily}

\newcommand{\CC}{\mathbb{C}} \newcommand{\NN}{\mathbb{N}}
\def\PP{{\mathbb{P}}} \def\RR{{\mathbb{R}}} \newcommand{\ZZ}{\mathbb{Z}} 
\def\Res{\mathcal{R}}

\newcommand{\sep}{,\xspace}

\begin{document}

\title{Sparse implicitization by interpolation: \\Geometric computations using matrix representations}
\author{Ioannis~Z.~Emiris\thanks{Dept Informatics and Telecoms, University of Athens, Greece.} \and
 Tatjana Kalinka$^*$  \and  
 Christos Konaxis$^*$\thanks{Corresponding author. E-mail: ckonaxis@di.uoa.gr, Tel.: +302107275342}}

\maketitle
 
\begin{abstract}
Based on the computation of a superset of the implicit support,
implicitization of a parametrically given (hyper)surface is reduced to computing the nullspace of a numeric matrix.
Our approach exploits the sparseness of the given parametric equations
and of the implicit polynomial.
In this work, we study how this interpolation matrix can be used
to reduce some key geometric predicates on the (hyper)surface
to simple numerical operations on the matrix,
namely membership and sidedness for given query points.
We illustrate our results with examples based on our Maple implementation.

\bigskip\noindent
\textit{Key\;words:}\/
geometric representation\sep implicitization\sep linear algebra\sep
sparse polynomial\sep membership\sep sidedness operation
\end{abstract}

\section{Introduction}

A fundamental question in changing representation of geometric objects is
implicitization, namely the process of changing the representation
of a geometric object from parametric to implicit. 
It is a fundamental operation with several applications
in computer-aided geometric design (CAGD) and geometric modeling. 
There have been numerous approaches for implicitization, including resultants,
Groebner bases, moving lines and surfaces, and interpolation techniques.

In this work, we restrict attention to hyper-surfaces and exploit a
matrix representation of hyper-surfaces in order
to perform certain critical operations efficiently, without developing
the actual implicit equation.
Our approach is based on potentially interpolating the unknown coefficients
of the implicit polynomial, but our algorithms shall avoid actually
computing these coefficients.
The basis of this approach is a sparse interpolation matrix, 
sparse in the sense that it is constructed 
when one is given a superset of the implicit polynomial's monomials.
The latter is computed by means of sparse resultant theory,
so as to exploit the input and output sparseness, in other words,
the structure of the parametric equations as well as the implicit 
polynomial. 

The notion that formalizes sparseness is the support of a polynomial
and its Newton polytope.  
Consider a polynomial $f$ with real coefficients in $n$ variables
$t_1,\dots,t_n$, denoted by
$$
f=\sum_a{c_{a}}{t}^{a}\, \in \RR[t_1,\dots,t_n], a\in\NN^n , c_{a}\in\RR,
\mbox{ where } t^a= t_1^{a_1}\cdots t_n^{a_n} .
$$
The \textit{support} of $f$ is the set $\{a \in \mathbb{N}^{n}: c_{a}\neq 0\};$
its \textit{Newton polytope} $N(f)\subset\RR^n$
is the convex hull of its support.
All concepts extend to the case of Laurent polynomials, i.e.\ with
integer exponent vectors $a\in\ZZ^n$.

We call the support and the Newton polytope of the implicit equation,
\emph{implicit support} and \emph{implicit polytope}, respectively. 
Its vertices are called \emph{implicit vertices}.
The implicit polytope is computed from the Newton polytope of the 
sparse (or toric) resultant, or \emph{resultant polytope},
of polynomials defined by the parametric equations.
Under certain genericity assumptions, the implicit polytope coincides
with a projection of the resultant polytope, see Section~\ref{Ssupport}.
In general, a translate of the implicit polytope is contained in the projected
resultant polytope,
in other words, a superset of the implicit support is given by the
lattice points contained in the projected resultant polytope, modulo the translation. 
A superset of the implicit support can also be obtained by other methods,
see Section~\ref{Sprevious}; the rest of our approach does not
depend on the method used to compute this support.

The predicted support is used to build a numerical matrix whose kernel
is, ideally, 1-dimensional, thus yielding (up to a nonzero scalar multiple)
the coefficients corresponding to the predicted implicit support.
This is a standard case of {\em sparse interpolation} of the polynomial
from its values. When dealing with hyper-surfaces of high dimension, or when the support
contains a large number of lattice points, then exact solving is expensive.
Since the kernel can be computed numerically, our approach also yields an
approximate sparse implicitization method.

Our method of sparse implicitization, which relies on interpolation,
was developed in \cite{EKKL11,EKKL12-spm}.
For details see the next section.
Our method handles (hyper)surfaces given parametrically by
polynomial, rational, or trigonometric parameterizations and, 
furthermore, automatically handles the case of base points.

The standard version of the method requires 
to compute the monomial expansion of the implicit equation.
However, it would be preferable if various operations and predicates
on the (hyper)surface could be completed by using the matrix without 
developing the implicit polynomial.
This is an area of strong current interest, since expanding, storing
and manipulating the implicit equation can be very expensive,
whereas the matrix offers compact storage and fast, linear algebra
computations.
This is precisely the premise of this work.

The main contribution of this work is to show that matrix representation
can be very useful when based on sparse interpolation matrices, which, when non-singular, have the 
property that their determinant equals the implicit equation.
In particular, we use the interpolation matrix to reduce
some geometric problems to numerical linear algebra.
We reduce the membership test
$p(q)=0$, for a query point $q$ and a hyper-surface defined implicitly
by $p(x)=0$, to a rank test on an
interpolation matrix for $p(x)$.
Moreover, we use the (nonzero) sign of the determinant of the same matrix
to decide sidedness for query points $q$ that do not lie on the surface
$p(x)=0$.
Our algorithms have been implemented in Maple.

The paper is organized as follows:
Section~\ref{Sprevious} overviews previous work.
Section~\ref{Ssupport} describes our approach to predicting
the implicit support while exploiting sparseness,
presents our implicitization algorithm based on
computing a matrix kernel and focuses on the case of high dimensional kernels.
In Section \ref{Soperations} we formulate membership and sidedness tests as numerical linear algebra operations
on the interpolation matrix.
We conclude with future work and open questions.

\section{Previous work}\label{Sprevious}

This section overviews existing work.

If $S$ is a superset of the implicit support, then the most direct method to reduce implicitization to linear algebra is
to construct a $|S| \times |S|$ matrix $M$,
indexed by monomials with exponents in $S$ (columns) and $|S|$ different values
(rows) at which all monomials get evaluated.
Then the vector 
of coefficients of the implicit equation is in the kernel of $M$. 
This idea was used in \cite{EKKL11,EK03,MarMar02,StuYu08}; it is also the
starting point of this paper.

An interpolation approach was 
based on integrating matrix $M=S S^\top$, over each parameter
$t_1, \dots , t_n$ \cite{CGKW00}. 
Then the vector of implicit coefficients is in the kernel of $M$. 
In fact, the authors propose to consider successively larger supports
in order to capture sparseness. 
This method covers polynomial, rational, and trigonometric
parameterizations, but the matrix entries take big values (e.g.\ up to $10^{28}$), so it is
difficult to control its numeric corank, i.e.\ the dimension of its nullspace. 
Thus, the accuracy of the approximate implicit polynomial is unsatisfactory.
When it is computed over floating-point numbers, the implicit polynomial
does not necessarily have integer coefficients.
They discuss post-processing to yield
integer relations among the coefficients, but only in small examples.

Our method of sparse implicitization was introduced in~\cite{EKKL11},
where the overall algorithm was presented together with some results on
its preliminary implementation, including the case
of approximate sparse implicitization. 
The emphasis of that work was on sampling and oversampling the parametric
object so as to create a numerically stable matrix, and examined
evaluating the monomials at random integers,
random complex numbers of modulus~1, and complex roots of unity.
That paper also proposed ways to obtain a smaller implicit polytope by
downscaling the original polytope when
the corresponding kernel dimension was higher than one.

One issue was that the kernel of the matrix might be of high dimension,
in which case 
the equation obtained may be a multiple of the implicit equation.
In~\cite{EKKL12-spm} this problem was addressed by
describing the predicted polytope and showing that, if 
the kernel is not 1 dimensional,
then the predicted polytope is the Minkowski sum
of the implicit polytope and an extraneous polytope.
The true implicit polynomial can be obtained by
taking the greatest common divisor (GCD)
of the polynomials corresponding to at least two and at most all
of the kernel vectors, or via multivariate polynomial factoring.
A different approach would be to Minkowski decompose the predicted 
polytope and identify its summand corresponding to the implicit polytope.
Our method handles (hyper)surfaces given parametrically by
polynomial, rational, or trigonometric parameterizations and, 
furthermore, automatically handles the case of base points.

Our implicitization method is based on the computation of the
implicit polytope, given the Newton polytopes of the parametric polynomials.
Then the implicit support is a subset of the set of
lattice points contained in the computed implicit polytope.
There are methods for the computation of the implicit polytope 
based on tropical geometry~\cite{StTeYu07,StuYu08}, see also~\cite{DAndSomb09}.
Our method relies on sparse elimination theory so as to compute
the Newton polytope of the sparse resultant.
In the case of curves, the implicit support is directly determined 
in~\cite{EKP10}.  

In~\cite{EFKP13ijcga}, they develop
an incremental algorithm to compute the resultant polytope, 
or its orthogonal projection along a given direction.
It is implemented in package
\texttt{ResPol}\footnote{\texttt{http://sourceforge.net/projects/respol}}.
The algorithm exactly computes vertex- and halfspace-representations
of the target polytope and it is output-sensitive.  
It also computes a triangulation of the polytope, which may be useful in
enumerating the lattice points.
It is efficient for inputs relevant to implicitization:
it computes the polytope of surface equations within $1$ second,
assuming there are less than 100 terms in the parametric polynomials,
which includes all common instances in geometric modeling.
This is the main tool for support prediction used in this work.

Approximate implicitization over floating-point numbers was introduced 
in a series of papers. Today, there are direct \cite{DT03,WTJD04}
and iterative techniques \cite{AiJuPo11}.
An idea used in approximate implicitization is to use
successively larger supports, starting with a quite small set and
extending it so as to reach the exact implicit support.
Existing approaches have used upper bounds on the
total implicit degree, thus ignoring any sparseness structure.
Our methods provide a formal manner to examine different supports,
in addition to exploiting sparseness, based on the implicit polytope.
When the kernel dimension is higher than one, one may downscale the
polytope so as to obtain a smaller implicit support.


The use of matrix representations in geometric modeling and CAGD is not new. 
In \cite{BMT09} they introduce matrix representations of algebraic curves and surfaces
and manipulate them  to address the 
curve/surface intersection problem by means of numerical linear algebra techniques
In \cite{BuTha12}, the authors make use of some generalized matrix-based
representations of parameterized surfaces in order to represent the 
intersection curve of two such surfaces as the zero set of a matrix determinant.
Their method extends to a larger class of rational 
parameterized surfaces, the applicability of a general approach to 
the surface/surface intersection problem in \cite{CuMa91}.
In \cite{Buse14} they introduce a new implicit representation of rational
B\'ezier curves 
and surfaces in the 3-dimensional space. Given such a curve or surface, this 
representation consists of a matrix whose entries depend on the space variables 
and whose rank drops exactly on this curve or surface.

\section{Implicitization by support prediction}\label{Ssupport} 

This section describes how sparse elimination can be used to compute the 
implicit polytope by exploiting sparseness 
and how this can reduce implicitization to linear algebra.
We also discuss how the quality of the predicted support affects 
the implicitization algorithm and develop the necessary constructions that 
allow us to formulate the membership and sidedness criteria in the next section.

\subsection{Sparse elimination and support prediction}

A \textit{parameterization} of a geometric object of co-dimension one,
in a space of dimension $n + 1$,
can be described by a set of parametric functions:
$$
{x_0} = f_0(t_1,\ldots,t_n),\ldots,x_n=f_n(t_1,\ldots,t_n),\, : \Omega \rightarrow \RR^{n+1},\,  \Omega:=\Omega_1\times\dots\times\Omega_n,~ \Omega_i \subseteq \RR
$$
where $t:=({t_1}, {t_2},\dots, {t_n})$ is the vector of parameters and
$f:=(f_0,$ $\dots,f_n)$ is a vector of continuous functions,
also called \textit{coordinate functions},
including polynomial, rational, and trigonometric functions.
We assume that, in the case of trigonometric functions, they may be converted to rational
functions by the standard half-angle transformation
$$
\sin \theta=\frac{2\tan {\theta}/{2} }{1+\tan^2 {\theta}/{2}},\;
\cos \theta=\frac{1-\tan^2 {\theta}/{2}}{1+\tan^2 {\theta}/{2}},
$$
where the parametric variable becomes $t=\tan \theta/2$.
On parameterizations depending on both $\theta$ and its trigonometric function, we may
approximate the latter by a constant number of terms in their series expansion.

The \textit{implicitization problem} asks for the smallest algebraic variety 
containing the closure of the image of the parametric map
$f :\RR^n\rightarrow\RR^{n+1}:  t \mapsto f(t)$.
Implicitization of planar curves and surfaces in three dimensional space
corresponds to $n = 1$ and $n = 2$ respectively. 
The image of $f$ is contained in the variety defined by the ideal of all polynomials
$p(x_0,\ldots,x_n)$ such that ${p}(f_0(t),\dots,f_n(t))={0}$, for all $t$ in $\Omega.$
We restrict ourselves to the case when this is a principal ideal, and we wish
to compute its unique defining polynomial 
$p(x_0,\dots,x_n) \in \RR[x_0,\ldots,x_n]$,
given its Newton polytope $P=N(p)$, or a polytope that contains it.
We can regard the variety in question as the (closure of the) projection of the graph of
map $f$ to the last $n+1$ coordinates.
Assuming the rational parameterization
\begin{equation}\label{eq:rational_system}
 {x_i}= {{f_i}(t)}/{{g_i}(t)}, \; i=0,\ldots,n ,
\end{equation}
implicitization is reduced to eliminating $t$ from
the polynomials in $(\RR[x_0,\ldots,x_n])[t,y]$:
\begin{align}\label{eq:polysystem}
F_i &:= {x_i}{{g_i}(t)} - {f_i}(t),~i=0,\ldots,n ,\\
\nonumber F_{n+1}&:=1- y g_0(t)\cdots g_n(t),
\end{align}
where $y$ is a new variable and $F_{i+1}$ assures that all $g_i(t)\ne 0$.
If one omits $F_{n+1}$, the generator of the corresponding (principal) ideal
would be a multiple of the implicit equation.
Then the extraneous factor corresponds to the $g_i$. 
Eliminating $t,y$ may be done by taking the \emph{resultant} of
the polynomials in~\eqref{eq:polysystem}.

Let $A_i\subset\ZZ^n,~i=0,\ldots,n+1$ be the supports of the polynomials $F_i$ and
consider the generic polynomials
\begin{equation}\label{eq:genpolysystem}
F_0',\dots,F_n',F_{n+1}' 
\end{equation}
with the same supports $A_i$ and symbolic coefficients ${c}_{ij}$.

\begin{definition}
Their \textit{sparse resultant} $\mbox{Res}(F_0',\ldots,F_{n+1}')$ is
a polynomial in the $c_{ij}$ with integer coefficients, namely
$$
\Res\in\mathbb{Z}[c_{ij}:i =0,\dots,n+1, j =1,\dots, |A_i| ] ,
$$
which is unique up to sign and vanishes if and only if the system
$F_0'=F_1'=\cdots=F_{n+1}'=0$ has a common root in a specific variety.
This variety is the projective variety $\PP^n$ over the algebraic closure
of the coefficient field
in the case of projective (or classical) resultants, or the toric variety
defined by the $A_i$'s. 
\end{definition}

The implicit equation of the parametric hyper-surface defined in~\eqref{eq:polysystem}
equals the resultant $\mbox{Res}(F_0,\ldots,F_{n+1})$, provided that the
latter does not vanish identically.
$\mbox{Res}(F_0,\ldots,F_{n+1})$ can be obtained from $\mbox{Res}(F_0',\ldots,F_{n+1}')$
by specializing the symbolic
coefficients of the $F_i'$'s to the actual coefficients of the $F_i$'s,
provided that this specialization is generic enough.
Then the implicit polytope $P$ equals the projection $Q$ of the 
resultant polytope to the space of the implicit variables,
i.e.\ the Newton polytope of the specialized resultant, up to some translation.
When the specialization of the $c_{ij}$ is not generic enough, then 
$Q$ contains a translate of $P$.
This follows from the fact that the method
computes the same resultant polytope as the tropical approach,
where the latter is specified in \cite{StTeYu07}.
Note that there is no exception even in the presence of base points.

\begin{proposition}{\rm\cite[Prop.5.3]{StTeYu07}}\label{Pnongeneric}
Let $f_0,\dots,f_n\in\CC[t_1^{\pm 1},$ $\dots, t_n^{\pm 1}]$ be any Laurent
polynomials whose ideal $I$ of algebraic relations is principal, say
$I=\langle p \rangle$, and let $P_i\subset\RR^n$ be the Newton polytope
of $f_i$.
Then the resultant polytope which is constructed combinatorially from 
$P_0,\dots,P_n$ contains a translate of the Newton polytope of $p$.
\end{proposition}

Our implicitization method is based on the computation of the
implicit polytope, given the Newton polytopes of the polynomials 
in~\eqref{eq:polysystem}.
Then the implicit support is a subset of the set of
lattice points contained in the computed implicit polytope.
In general, the implicit polytope is obtained from the
projection of the resultant polytope of the polynomials
in~\eqref{eq:genpolysystem} defined by the specialization of their
symbolic coefficients to those of the polynomials in~\eqref{eq:polysystem}.
For the resultant polytope we employ \cite{EFKP13ijcga} and software 
\texttt{ResPol}\footnote{\texttt{http://sourceforge.net/projects/respol}}.

Sparse elimination theory works over the ring of Laurent polynomials $\CC[t_1^{\pm 1},$ $\dots, t_n^{\pm 1}]$
which means that points in the supports of the polynomials 
may have negative coordinates. As a consequence, evaluation points of polynomials cannot have zero coordinates.
In the rest of the paper we make the assumption that all 
polytopes are translated to the positive orthant and have non-empty intersection with every
coordinate axis. This allows us to consider points with zero coordinates.

\subsection{Matrix constructions}\label{SSmatconstruct}

The predicted implicit polytope $Q$ and the set $S$ of lattice points it contains are
used in our implicitization algorithm to construct 
a numerical matrix $M$.
The vectors in the kernel of $M$ contain the coefficients of the monomials
with exponents in $S$ in the implicit polynomial $p(x)$.

Let us describe this construction.
Let $S:=\{ s_1,\ldots,s_{|S|}\}$; each
$s_j = (s_{j0},\ldots,s_{jn})$, $j=1,\ldots,|S|$ is an exponent of a (potential) monomial
$m_j:=x^{s_j}=x_0^{s_{j0}} \cdots x_n^{s_{jn}}$  
of the implicit polynomial,
where ${x_i}= {{f_i}(t)}/{{g_i}(t)}, \; i=0,\ldots,n$, as in
\eqref{eq:rational_system}.
We evaluate $m_j$ at some \emph{generic} point 
$\tau_k \in \CC^n, ~ k = 1,\dots, \mu,~ \mu\geq |S|$. 
Let 
\[
m_j |_{t=\tau_k} :=
\prod_{i=0}^{n} \left(\frac{{f_i}(\tau_k)}{{g_i}(\tau_k)}\right)^{s_{ji}}, \quad j=1,\ldots,|S|
\]
denote the evaluated $j$-th monomial $m_j$ at $\tau_k$.
Thus, we construct an $\mu\times |S|$ matrix $M$
with rows indexed by $\tau_1,\ldots,\tau_\mu$ and columns by $m_1,\ldots,m_{|S|}$:
\begin{equation}\label{Eimatrix}
 M = \begin{bmatrix}
m_1 |_{t=\tau_1} & \cdots & m_{|S|} |_{t=\tau_1}\\[3pt]
\vdots                     & \cdots &  \vdots   \\[3pt]
m_1 |_{t=\tau_{\mu}} & \cdots & m_{|S|} |_{t=\tau_{\mu}}      
\end{bmatrix}.
\end{equation}

To cope with numerical issues, especially when computation is approximate,
we construct a rectangular matrix $M$ by choosing $\mu \geq |S|$ 
values of $\tau$; this overconstrained system increases numerical stability.
Typically $\mu=|S|$ for performing exact kernel computation, and
$\mu= 2|S|$ for approximate numeric computation.

When constructing matrix $M$ we make the assumption that the parametric hyper-surface 
is sampled sufficiently generically by evaluating the parametric expressions at
random points $\tau_k \in \CC^{n}$. Hence we have the following:

\begin{lemma}[\cite{EKKL12-spm}]\label{Lmultiple}
Any polynomial in the basis of monomials $S$ indexing $M$, with coefficient
vector in the kernel of $M$, is a multiple of the implicit polynomial $p(x)$. 
\end{lemma}

As in~\cite{EKKL11}, one of the main difficulties is to build $M$ whose
corank, or kernel dimension, equals~1, i.e.\
its rank is~1 less than its column dimension.
Of course, we avoid values that make the denominators of the
parametric expressions close to~0.

For some inputs we obtain a matrix of $\mbox{corank}>1$
when the predicted polytope $Q$ is significantly larger than $P$.
It can be explained by the nature of our method: we rely on a
\emph{generic} resultant to express the implicit equation,
whose symbolic coefficients are then specialized to the actual coefficients of the parametric equations. 
If this specialization is not generic,
then the implicit equation divides the specialized resultant.  
The following theorem establishes the relation between the dimension
of the kernel of $M$ and the accuracy of the predicted support.
It remains valid even in the presence of base points.
In fact, it also accounts for them since then $P$ is expected to be
much smaller than $Q$. 

\begin{theorem}[\cite{EKKL12-spm}]\label{Tcorank1}
Let $P=N(p)$ be the Newton polytope of the implicit equation,
and $Q$ the predicted polytope.
Assuming $M$ has been built using sufficiently generic evaluation points,
the dimension of its kernel equals $r=\#\{ a \in \ZZ^{n+1} : a+P \subseteq Q \} =
\#\{ a \in \ZZ^{n+1} : N(x^a \cdot p) \subseteq Q \}$. 
\end{theorem}

The formula for the corank of the matrix in the previous theorem 
also implies that the coefficients of the polynomials
$x^{a}p(x)$ such that $N(x^ap(x)) \subseteq Q$, form a basis of the kernel of $M$
(see \cite[Proof of Thm.~10]{EKKL12-spm}).
This observation will be useful in the upcoming Lemma~\ref{Lmembership}
but also implies the following corollary.

%
\begin{corollary}[\cite{EKKL12-spm}]\label{Cigcd}
Let $M$ be the matrix from \eqref{Eimatrix}, built with sufficiently
generic evaluation points, and suppose the specialization of the
polynomials in~\eqref{eq:genpolysystem} to the
parametric equations is sufficiently generic.
Let $\{\bm{c_1},\ldots, \bm{c_\lambda}\}$ be a basis of the kernel of $M$ and
$g_1(x),\ldots, g_\lambda(x)$ be the polynomials 
obtained as the inner product $g_i=\bm{c_i}\cdot \bm{m}$.
Then the greatest common divisor (GCD) of $g_1(x),\ldots, g_\lambda(x)$ equals the implicit equation
up to a monomial factor $x^e$.
\end{corollary}

\begin{remark}\label{Rextrmonom}
The extraneous monomial factor $x^e$ in the previous corollary 
is always a constant when the predicted polytope $Q$ is of the form 
$Q=P+R$ and, as we assume throughout this paper, it is translated to the positive orthant
and touches the coordinate axis.
However, it is possible that $Q$ strictly contains $P+R$
and the extraneous polytope $R$ is a point $e\in \RR^{n+1}$,  
or it is the Minkowski sum of point $e$ and a polytope $R'$ which touches the axis.
Let $\sum_\beta c_\beta x^\beta$ be the GCD of the polynomials $g_i$ in Corollary \ref{Cigcd},
and let $\gamma=(\gamma_0,\ldots,\gamma_n)$, where $\gamma_i=\min_\beta(\beta_i)$, $i=0,\ldots,n$.
We can efficiently remove the extraneous monomial $x^e$ by 
dividing $\sum_\beta c_\beta x^\beta$ with $x^\gamma$, i.e. the GCD of monomials $x^\beta$.
\end{remark}

We modify slightly the construction of the interpolation matrix $M$ to 
obtain a matrix denoted $M(x)$ which is numeric except for its last row.
This matrix will be crucial in formulating our geometric predicates.

Fix a set of \emph{generic} distinct values $\tau_k, ~ k = 1,\dots,|S|-1$ and recall that 
$m_j |_{t=\tau_k}$
denotes the $j$-th monomial $m_j$ evaluated at $\tau_k$.
Let $M'$ be the $|S|-1\times{}|S|$ numeric matrix obtained by 
evaluating $\bm{m}$ at $|S|-1$ points $\tau_k,\, k=1\ldots,|S|-1$:
\begin{equation}\label{Etempmatrix}
 M' = \begin{bmatrix}
m_1 |_{t=\tau_1} & \cdots & m_{|S|} |_{t=\tau_1}\\[3pt]
\vdots                     & \cdots &  \vdots   \\[3pt]
m_1 |_{t=\tau_{|S|-1}} & \cdots & m_{|S|} |_{t=\tau_{|S|-1}} 
\end{bmatrix},
\end{equation}
and $M(x)$ be the $|S|\times{}|S|$ matrix obtained by appending  the row vector $\bm{m}$ 
to matrix $M'$:
\begin{equation}\label{Ematx}
M(x) = \begin{bmatrix}
M' \\[2pt]
\bm{m}
\end{bmatrix}.
\end{equation}

Algorithm~\ref{Amatrixx} summarizes the construction of matrix $M(x)$.
\begin{algorithm}[h] \label{Amatrixx}
\caption{Matx} 
\dontprintsemicolon
\Input{Parameterization
$x_i= f_i(t)/g_i(t),\, i=0,\dots,n$,\\
Predicted implicit polytope $Q$.\\
}
\Output{Matrix $M(x)$}

\medskip

$\NN^{n+1} \supseteq S \leftarrow \mbox{lattice points in}~Q$\;
\lForEach{$s_i \in S$}{$m_i \leftarrow x^{s_i}$} \hspace{1em} \tcp{$x:=(x_0,\ldots,x_n)$}\;
$\bm{m} \leftarrow (m_1,\ldots,m_{|S|})$  \hspace{1em} \tcp{vector of monomials in $x$}\; 
Initialize  $|S|-1 \times |S|$ matrix $M'$: \;
\For{$i \leftarrow 1$ \KwTo $|S|-1$}{
select~$\tau_i\in\CC^{n+1}$ \\
\For{$j \leftarrow 1$ \KwTo $|S|$}{ $M_{ij} \leftarrow m_j|_{t=\tau_i}$}
}
Append row $\bm{m}$ to $M'$: \;
$M(x) \leftarrow \begin{bmatrix} M' \\[1pt] \bm{m} \end{bmatrix}$\;
\smallskip
\Return $M',\, M(x)$\;
\end{algorithm}

Given a point $q\in \RR^{n+1}$, let $M(q)$ be the matrix 
$M(q) = \begin{bmatrix} 
M' \\[2pt]
\bm{m}|_{x=q}
\end{bmatrix}$,
where $\bm{m} |_{x=q}$ denotes the vector $\bm{m}$ of predicted monomials evaluated at point $q$.
We assume that $q$ does not coincide with any 
of the points $x(\tau_k),\, k=1,\ldots,|S|-1$ used to build matrix $M'$,
which implies that the rows of matrix $M(q)$ are distinct.
This can be checked efficiently. Obviously, when $q$ lies on the hyper-surface $p(x)=0$,
matrix $M(q)$ is equivalent to matrix $M$ in \eqref{Eimatrix} 
in the sense that they both have the same kernel.

\begin{remark}\label{Rsamerank}
 Let $M$ be a matrix as in \eqref{Eimatrix} and $\bm{c}$ be a vector in the kernel of $M$.
 Then the vector $\lambda\bm{c}$, for any $0\neq \lambda \in \RR$, is also in the kernel of $M$ because
 the implicit polynomial is defined up to a non-zero scalar multiple.
 This means that we can set an arbitrary coordinate of $\bm{c}$ equal to 1.
 As a consequence the matrices $M'$ and $M$ (and from the discussion above also $M(q)$, where $p(q)=0$),
 have the same kernel of corank $r$, where $r$ is given in Theorem \ref{Tcorank1}.
\end{remark}

Matrix $M(x)$ has an important property as shown in the following
\begin{lemma}\label{Lmatrixx}
Assuming $M'$ is of full rank, the determinant of matrix $M(x)$ equals 
the implicit polynomial $p(x)$ up to a constant.
\end{lemma}
\begin{proof}
Suppose that $M'$ is of full rank equal to $|S|-1$. Then there exists a non-singular $(|S|-1)\times(|S|-1)$ 
submatrix of $M'$. Without loss of generality we assume that is is the submatrix $M''=M'_{-|S|}$
obtained from $M'$ by removing its last column. 
By Remark \ref{Rsamerank},
$M'$ and $M$ have the same kernel
consisting of a single vector $\bm{c}=(c_1,\ldots,c_{|S|})$, where we
can assume that $c_{|S|}=1$. Then
\begin{align}\label{Ematxproof}
\nonumber
M'\cdot \begin{bmatrix} c_1\\ \vdots \\ c_{|S|-1}\\ 1 \end{bmatrix} = \bm{0}
\Leftrightarrow
\begin{bmatrix}
m_1 |_{t=\tau_1} & \cdots & m_{|S|} |_{t=\tau_1}\\[3pt]
\vdots                     & \cdots &  \vdots   \\[3pt]
m_1 |_{t=\tau_{|S|-1}} & \cdots & m_{|S|} |_{t=\tau_{|S|-1}} 
\end{bmatrix} 
\cdot 
\begin{bmatrix} c_1\\ \vdots \\ c_{|S|-1}\\ 1 \end{bmatrix}= \bm{0}
\Leftrightarrow &\\
\begin{bmatrix}
m_1 |_{t=\tau_1} & \cdots & m_{|S|-1} |_{t=\tau_1}\\[3pt]
\vdots                     & \cdots &  \vdots   \\[3pt]
m_1 |_{t=\tau_{|S|-1}} & \cdots & m_{|S|-1} |_{t=\tau_{|S|-1}} 
\end{bmatrix} 
\cdot 
\begin{bmatrix} c_1\\ \vdots \\ c_{|S|-1}\\ 1 \end{bmatrix}
= 
-\begin{bmatrix}
m_{|S|} |_{t=\tau_1}\\
\vdots \\
m_{|S|} |_{t=\tau_{|S|-1}}
\end{bmatrix},
\end{align}
which, by applying Cramer's rule yields 
\begin{equation}\label{Ecvec}
 c_k=\frac{\det M''_{k}}{\det M''}, \quad k=1,\ldots,|S|-1,
\end{equation}
where $M''_k$ is the matrix obtained by replacing the $k$th column of $M''$ by the $|S|$th column of $M'$,
which plays the role of the constant vector in \eqref{Ematxproof}. 
Note that $M''_k$ equals (up to reordering of the columns)  $M'_{-k}$,  where  $M'_{-k}$ is the matrix 
obtained by removing the $k$th column of $M'$. Hence, 
$\det M''_k$ equals (up to sign)  $\det M'_{-k}$.

Now, the assumption that $M'$ is of full rank in conjunction with Theorem \ref{Tcorank1} 
and Corollary \ref{Cigcd} implies that 
\begin{align*}
 p(x)=\bm{m}\cdot\bm{c}=\sum_{i=1}^{|S|} m_i \cdot c_i = \sum_{i=1}^{|S|-1} m_i \cdot c_i + m_{|S|},
\end{align*}
which combined with \eqref{Ecvec} gives
\begin{align*}
 p(x)&= \sum_{i=1}^{|S|-1} m_i \cdot \frac{\det M''_{k}}{\det M''} + m_{|S|}=
 \pm\sum_{i=1}^{|S|-1} m_i \cdot\frac{\det M'_{-k}}{\det M'_{-|S|}} + m_{|S|}\\
 &=\pm\frac{1}{\det M'_{-|S|}}\sum_{i=1}^{|S|-1} m_i \cdot \det M'_{-k} + m_{|S|}=
 \pm\frac{1}{\det M'_{-|S|}}\sum_{i=1}^{|S|} m_i \cdot \det M'_{-k} \\
 &=  \pm\frac{\det M(x)}{\det M'_{-|S|}}. 
\end{align*}
\end{proof}


\section{Geometric Operations}\label{Soperations}

In this section we formulate certain elementary geometric operations 
on the hyper-surface defined by $p(x)$
as matrix operations.
In particular, we focus on membership and sidedness.


\paragraph{Membership predicate.}

Given a parameterization $x_i=f_i(t)/g_i(t),~ i=0,\ldots,n$, 
and a query point $q\in \RR^{n+1}$, we want to decide if $p(q)=0$ is true or not,
where $p(x)$ is the implicit equation of the parameterized hyper-surface.
Our goal is to formulate this test using the interpolation matrix in \ref{Ematx}.

Working with matrices instead of polynomials, we cannot 
utilize Corollary \ref{Cigcd} and Remark \ref{Rextrmonom} to process the kernel 
polynomials.
Thus, a kernel polynomial might be of the form $x^e p(x)$.
To avoid reporting a false positive when evaluating such a polynomial
at a query point having zero coordinates, 
we restrict membership testing to points $q\in (\RR^*)^{n+1}$, where $\RR^*=\RR\setminus\{0\}$.

\begin{lemma}\label{Lmembership}
Let $M(x)$ be as in \eqref{Ematx} and $q=(q_0,\ldots,q_n)$ be a query  point in $(\RR^*)^{n+1}$.
Then $q$ lies on the hyper-surface defined by $p(x)=0$ if and only if $\mbox{corank}(M(q))=\mbox{corank}(M')$. 
\end{lemma}
\begin{proof}
For every point $q$, since $M'$ is an $(|S|-1) \times |S|$ submatrix of
the $|S|\times|S|$ matrix $M(q)$,
we have that $\mbox{rank}(M(q)) \geq \mbox{rank}(M')$ 
which implies that
$\mbox{corank}(M(q)) \leq \mbox{corank}(M')$.
Moreover, it holds that
\begin{equation}\label{Ekernels}
\mbox{kernel}(M(q))\subseteq \mbox{kernel}(M').
\end{equation}
 
\smallskip
\noindent $(\rightarrow)$ 
 Assume that $q$ lies on the hyper-surface defined by $p$, hence $p(q)=0$.
 Then by Remark \ref{Rsamerank} the matrices $M(q)$ and $M'$ have the same corank.
 
\smallskip
\noindent $(\leftarrow)$ 
 Suppose that $\mbox{corank}(M(q)) = \mbox{corank}(M')$.
 Then the last row $\bm{m} |_{x=q}$ of $M(q)$ is linearly dependent on the first $|S|-1$ rows, 
 hence there exist $l_k\in \RR, k=1,\ldots,|S|$, not all zero,  such that
 $\bm{m}|_{x=q}=\sum_{k=1}^{|S|} l_k \bm{m} |_{t=\tau_k}$. 
 Let $\bm{c}\in \mbox{kernel}(M')$. Then
 $\bm{m}|_{x=q} \cdot \bm{c}=(\sum_{i=k}^{|S|} l_k \bm{m} |_{t=\tau_k})\cdot \bm{c}=
 \sum_{i=k}^{|S|} l_k (\bm{m} |_{t=\tau_k}\cdot \bm{c}) = 0$, so $\bm{c}\in \mbox{kernel}(M(q)$,
 which, given relation \eqref{Ekernels}, implies that $M'$ and $M(q)$ have the same kernel.
 
 Every vector $\bm{c}$ in the kernel of $M'$, hence, also of $M(q)$,  is a linear combination  of 
 the coefficient vectors of the polynomials $x^a p(x)$, 
 where $a\in \ZZ^{n+1}$ such that $N(x^a p(x))\subseteq Q$, 
 (see also the discussion following Theorem~\ref{Tcorank1}).
 So we have $\bm{m} |_{x=q}\cdot \bm{c}=\sum_a \lambda_a q^a p(q) =0$, 
 where $\lambda_a \in \RR$ are not all equal to zero,
 which, since $q\in (\RR^*)^{n+1}$, implies that $p(q)=0$.
\end{proof}

Lemma~\ref{Lmembership} readily yields Algorithm~\ref{Amembership} that reduces the membership
test $p(q)=0$ for a query point $q\in \RR^{n+1}$, to the comparison of the ranks of the matrices $M'$ and $M(q)$.

\begin{algorithm}[ht] \label{Amembership}
\caption{Membership} 
\dontprintsemicolon
\Input{Parameterization
$x_i= f_i(t)/g_i(t),\, i=0,\dots,n$,\\
Predicted implicit polytope $Q$,\\
Query point $q\in (\RR^*)^{n+1}$.\;
}
\Output{0 if  $p(q)=0$, 1 otherwise.}

\medskip
$M', M(x) \leftarrow$ Matx$( f_i(t)/g_i(t),\, i=0,\dots,n;\,Q)$\;
\lIf{$\mbox{corank}(M(q))=\mbox{corank}(M')$}{$\alpha \leftarrow 0$}\;
\Else{$\alpha \leftarrow 1$}\;
\Return $\alpha$\;
\end{algorithm}

\paragraph{Sidedness predicate.}

Let us now consider the sidedness operation for the hyper-surface $p(x)=0$, which we define using 
the sign of the evaluated polynomial $p(q)$, for $q\in \RR$:
\begin{definition}\label{Dsidedness}
Given a hyper-surface in $\RR^{n+1}$ with defining equation $p(x)\in \RR[x]$, and a 
point $q\in \RR^{n+1}$ such that $p(q)\neq 0$, 
we define $\mbox{side}(q)=\mbox{sign}(p(q)) \in \{-1,1\}$.
\end{definition}

See Figure \ref{Ffolium} for an example of applying Definition \ref{Dsidedness} to 
the folium of Descartes curve defined by $y^3-3xy+x^3=0$.
\begin{figure}[ht] 
\includegraphics[scale=0.3]{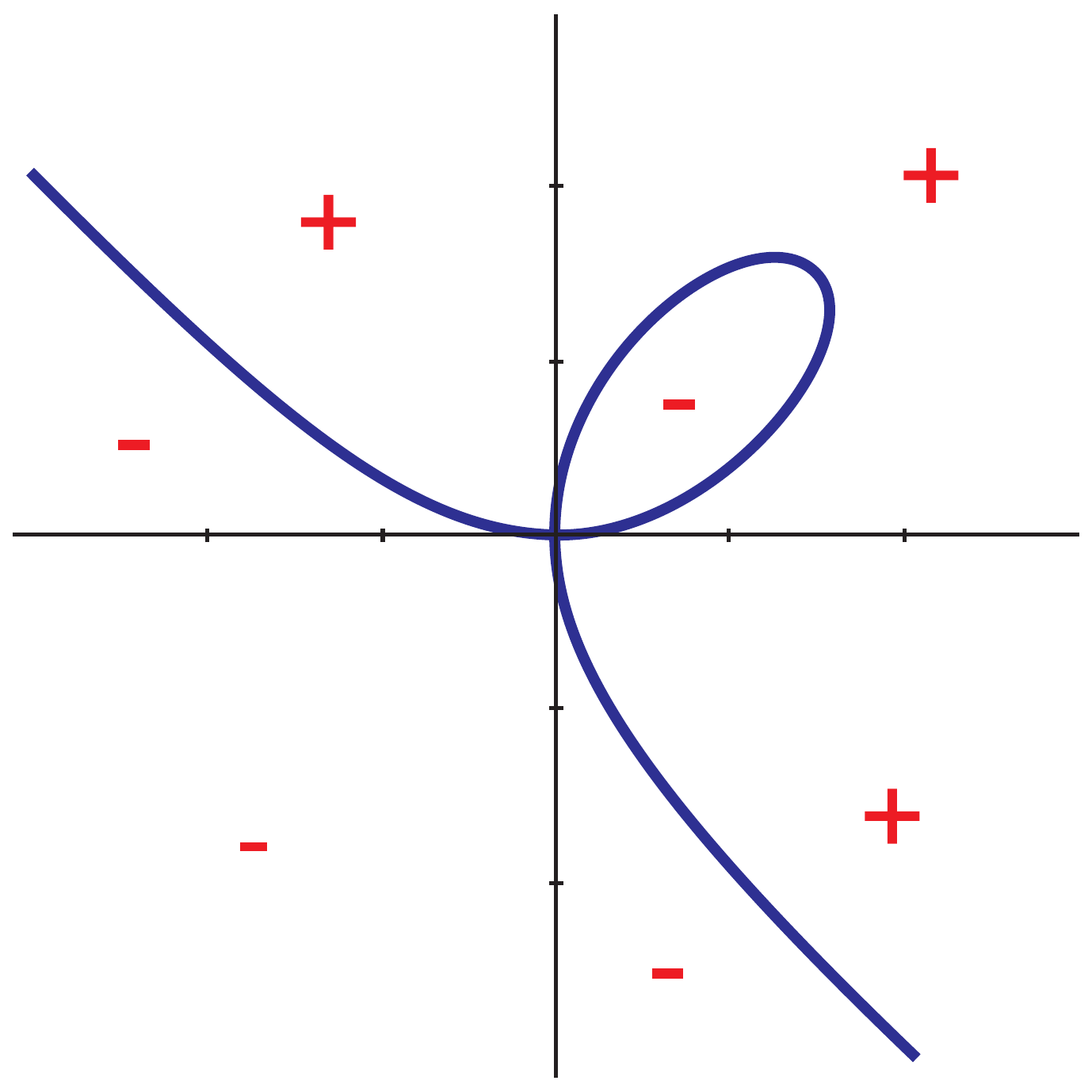}
\centering\caption{The sign of the polynomial defining the folium of Descartes.}
\label{Ffolium}
\end{figure} 

We will use matrix $M(x)$ defined in \eqref{Ematx} to 
reduce sidedness in the sense of Definition \ref{Dsidedness},
to the computation of the sign of a numerical determinant.
First we show that this determinant is non-zero for relevant inputs.

\begin{lemma}\label{Laux-sided}
 Suppose that the predicted polytope $Q$ contains only one translate of the implicit polytope $P$. 
 Let $M(x)$ be a matrix as in \eqref{Ematx} and let $q\in (\RR^*)^{n+1}$ such that $p(q)\neq 0$. 
 Then $\det M(q) \neq 0$.
\end{lemma}

\begin{proof}
 Since the predicted polytope $Q$ contains only one translate of the implicit polytope $P$, 
 Theorem \ref{Tcorank1} implies that  $\mbox{corank}(M)=1$ and by Remark \ref{Rsamerank} 
 this means that $\mbox{corank}(M')=1$, where matrices $M,M'$ 
 are defined in \eqref{Eimatrix},\eqref{Etempmatrix} respectively.
 Then since $p(q)\neq 0$, from Lemma \ref{Lmembership} we have that
 $\mbox{corank}(M')\neq \mbox{corank}(M(q)$, which implies that
 the matrix $M(q)$ is of full rank equal to $|S|$.
 Hence $\det M(q)\neq 0$.
\end{proof} 

Next we show that, given matrix $M(x)$ and a point $q\in (\RR^*)^{n+1}$ such that $p(x)\neq 0$,
the sign of $\det(M(q))$ is consistent with $\mbox{side}(q)$ in the following sense:
for every pair of query points $q_1,q_2$, whenever $\mbox{side}(q_1)=\mbox{side}(q_2)$, 
we have that $\mbox{sign}(\det M(q_1))=\mbox{sign}(\det M(q_2))$.

The following theorem is the basic result for our approach. 

\begin{theorem}
Let $M(x)$ be as in \eqref{Ematx} and $q_1, q_2$ be two query  points in $(\RR^*)^{n+1}$ 
not lying on the hyper-surface defined by $p(x)=0$.
Assuming that the predicted polytope $Q$ contains only one translate of the implicit polytope $P$,
then  $\mbox{side}(q_1)=\mbox{side}(q_2)$
if and only if $\mbox{sign}(\det M(q_1))=\mbox{sign}(\det M(q_2))$,
where sign$(\cdot)$ is an integer in $\{-1,1\}$. 
\end{theorem}

\begin{proof}
For points $q_1, q_2$ as in the statement of the theorem, 
we have from Lemma \ref{Laux-sided} that $\det M(q_1))$ and $\det M(q_2))$
are non-zero, hence their sign is an integer in $\{-1,1\}$. 
We need to show that $\mbox{sign}(p(q_1))=\mbox{sign}(p(q_2))$ if and only if 
$\mbox{sign}(\det M(q_1))=\mbox{sign}(\det M(q_2))$.
But this is an immediate consequence from Lemma \ref{Lmatrixx}, since
$\det M(x)$ equals $p(x)$ up to a constant factor.
\end{proof}

Algorithm  \ref{Asidedness} summarizes the previous discussions
for deciding sidedness for any two query points.
The rank test at step 2 of the algorithm van be avoided if 
we directly compute sign$(\det M(q_i))$ and 
proceed depending on whether this sign equals 0 
(i.e. $\det  M(q_i) = 0$) or not.

\begin{algorithm}[ht] \label{Asidedness}
\caption{Sidedness} 
\dontprintsemicolon
\Input{Polynomial or rational parameterization
$x_i= f_i(t)/g_i(t),\, i=0,\dots,n$,\\
Predicted implicit polytope $Q$,\,
Query points $q_1, q_2\in (\RR^*)^{n+1}$.\;
}
\Output{1 if points lie on the same side of $p$,
0 if  $p(q_1)=0$ or $p(q_2)=0$, and -1 otherwise.}

\medskip
$M', M(x) \leftarrow$ Matx$(f_i(t)/g_i(t),\, i=0,\dots,n,\, Q)$\;
\lIf{$\mbox{corank}(M(q_1))=\mbox{corank}(M') \mbox{~or~} \mbox{corank}(M(q_2))=\mbox{corank}(M')$}{$\alpha \leftarrow 0$}\;
\lElseIf{$\mbox{sign}(\det M(q_1))=\mbox{sign}(\det M(q_2)$}{$\alpha \leftarrow 1$}\;
\Else{$\alpha \leftarrow -1$}\;
\Return $\alpha$\;
\end{algorithm}

\section{Conclusion}

We have shown that certain operations can be accomplished on the matrix 
representation of a hyper-surface, namely by using the interpolation matrix 
that we constructed. Our current work includes the study of further operations, 
most notably ray shooting, either in exact or approximate form.

To perform ray shooting, assume that a ray is parameterized 
by $\rho>0$ and given by $x_i=r_i(\rho),\, i=0,\ldots,n$, 
where the $r_i$ are linear polynomials in $\rho$.
We substitute $x_i$ by $r_i(\rho)$ in $M(x)$  thus obtaining a univariate 
matrix $M(\rho)$ whose entries are numeric except for its last row.
Assuming that $M(x)$ is not singular, 
imagine that we wish to develop $\det M(\rho)$ by expanding along its last row. 
Then in preprocessing we compute 
all minors $\det M_{|S|j}, \, j=1,\ldots,|S|$, corresponding to entries in the last row, where
 $M_{ij}$ is the submatrix of $M(\rho)$ obtained by deleting its $i$th row and $j$th column,
thus defining 
\[ 
p(\rho)=\sum_{j=1}^{|S|} (-1)^{|S|+j}\det (M_{|S|j}) m_j(\rho).              
\]
Note that every  $m_j(\rho)$ is a product of powers of linear polynomials in $\rho$.
Now ray shooting is reduced to finding the smallest positive real root 
of a univariate polynomial.
For this, we plan to employ state of the art real solvers which typically 
require values of the polynomial in hand. 

We are also extending our implementation so as to offer a 
complete and robust Maple software. At the same time we 
continue to develop code that uses specialized numerical 
linear algebra in \texttt{Sage} and \texttt{Matlab}.
More importantly, we shall continue comparisons to other methods 
which were started in \cite{EKKL11}.
 
Future work includes studying the matrix structure, which 
generalizes the classic Vandermonde structure, since
the matrix columns are indexed by monomials and the
rows by values on which the monomials are evaluated.
This reduces matrix-vector multiplication to multipoint evaluation
of a multivariate polynomial. 
However, to gain an order of magnitude in matrix operations 
we would have to implement fast multivariate interpolation and 
evaluation over arbitrary points, for which there are many open questions.

Sparse interpolation is the problem of interpolating a multivariate
polynomial when information of its support is given~\cite[Ch.14]{Zippel93}.  
This may simply be a bound $\sigma=|S|$ on support cardinality; then
complexity is $O( m^3 \delta n\log n + \sigma^3 )$, where $\delta$ bounds the
output degree per variable, $m$ is the actual support cardinality,
and $n$ the number of variables.
A probabilistic approach in $O(m^2 \delta n)$ requires as input only $\delta$.
The most difficult step in the type of algorithms by Zippel for interpolation is 
computing the multivariate support: in our case this is known, hence our task should be easier.

Lastly, we are developing algorithms for Minkowski decomposition in $\RR^2$ and $\RR^3$ 
which should significantly reduce the size of polytope $Q$ hence 
the number of lattice points that define $M$ when the predicated polytope 
is larger than the precise implicit polytope.

\paragraph*{Acknowledgement.}
This research has been co-financed by the European Union
(European Social Fund - ESF) and Greek national funds through
the Operational Program ``Education and Lifelong Learning" 
of the National Strategic Reference Framework (NSRF) - Research
Funding Program: THALIS-UOA (MIS 375891).

\bibliographystyle{alpha}
\bibliography{ekk14} 

\end{document}